\documentclass[a4paper]{amsart}
\usepackage[latin1]{inputenc}
\usepackage{amssymb}
\usepackage{amsmath}
\usepackage{mathrsfs}
\usepackage{amsthm}
\usepackage{amsfonts}
\usepackage{cancel}
\usepackage{textcomp}
\usepackage{graphicx}
\usepackage[pdftex]{color}
\usepackage{paralist}
\usepackage[shortlabels]{enumitem}
\usepackage{hyperref}
\usepackage{comment}
\usepackage{mathtools}
\usepackage[arrow, matrix, curve]{xy}

\newtheorem*{corollary*}{Corollary}
\newtheorem*{theorem*}{Theorem}
\newtheorem{theorem}{Theorem}[section]

\newtheorem{corollary}[theorem]{Corollary}
\newtheorem{lemma}[theorem]{Lemma}

\newtheorem{proposition}[theorem]{Proposition}
\newtheorem{question}[theorem]{Question}

\newtheorem*{claim*}{Claim}
\newtheorem*{conjecture}{Conjecture}

\theoremstyle{definition}

\newtheorem{remark}[theorem]{Remark}
\newtheorem{example}[theorem]{Example}

\theoremstyle{remark}

\numberwithin{equation}{theorem}

\makeatletter
\renewcommand*\env@matrix[1][\
arraystretch]{%
  \edef\arraystretch{#1}%
  \hskip -\arraycolsep
  \let\@ifnextchar\new@ifnextchar
  \array{*\c@MaxMatrixCols c}}
\makeatother

\renewcommand{\mod}{\operatorname{mod-} \!}

\newcommand{\domdim}{\operatorname{domdim}}
\newcommand{\sdomdim}{\operatorname{sdomdim}}
\newcommand{\Gdim}{\operatorname{Gdim}}
\newcommand{\Def}{\operatorname{Def}}
\newcommand{\id}{\operatorname{id}}
\newcommand{\pd}{\operatorname{pd}}
\newcommand{\findim}{\operatorname{findim}}
\newcommand{\scodomdim}{\operatorname{scodomdim}}

\newcommand{\gldim}{\operatorname{gldim}}
\newcommand{\codomdim}{\operatorname{codomdim}}
\newcommand{\Int}{\operatorname{Int}}
\newcommand{\End}{\operatorname{End}}
\newcommand{\Hom}{\operatorname{Hom}}

\renewcommand{\top}{\operatorname{\mathrm{top}}}

\newcommand{\soc}{\operatorname{\mathrm{soc}}}

\newcommand{\op}{\operatorname{op}}

\begin{document}

\title{On the classification of higher Auslander algebras for Nakayama algebras}
\date{\today}


\author{Dag Oskar Madsen}
\address{Faculty of Education and Arts, Nord University, Postboks 1490, NO-8049 Bod{\o}, Norway}
\email{dag.o.madsen@nord.no}
\author{Ren\'{e} Marczinzik$^*$}
\address{Institute of algebra and number theory, University of Stuttgart, Pfaffenwaldring 57, 70569 Stuttgart, Germany}
\thanks{\newline\indent$^*$ corresponding author}
\email{marczire@mathematik.uni-stuttgart.de}
\author{Gjergji Zaimi}
\email{gjergjiz@gmail.com}
\begin{abstract}
We give new improved bounds for the dominant dimension of Nakayama algebras and use those bounds to give a classification of Nakayama algebras with $n$ simple modules that are higher Auslander algebras with global dimension at least $n$.
The classification is then used to extend the results on the inequality for the global dimension of Nakayama algebras obtained in \cite{MM}.
\end{abstract}

\maketitle
\section*{Introduction}

We assume that all algebras are finite dimensional (bound) quiver algebras over a field $K$ and furthermore that they are connected and non-semisimple if nothing is stated otherwise.
Higher Auslander algebras were introduced by Iyama in \cite{Iya} and \cite{Iya2} as a generalisation of the classical Auslander algebras.
Higher Auslander algebras are defined as algebras with finite global dimension at least two such that the global dimension coincides with the dominant dimension, and in \cite{Iya2} it was shown that they correspond bijectively to cluster tilting objects in module categories.
It is a highly complicated problem to give a classification of higher Auslander algebras even for special classes of non-trivial algebras.
One of the most basic classes of algebras in the representation theory of finite dimensional algebras are the Nakayama algebras, which are defined by the condition that any indecomposable left or right projective module is uniserial.
Nakayama algebras appear in various places, we mention the classification of Brauer tree algebras up to stable equivalence see for example chapter X.3 of \cite{ARS}, the appearance of the derived categories of Nakayama algebras in algebraic geometry in \cite{KLM}, and a recent connection between Nakayama algebras and the combinatorics of Dyck paths in \cite{MRS}.
In the recent article \cite{JK}, higher Nakayama algebras were introduced as a generalisation of Nakayama algebras and used to give large new classes of examples for higher Auslander-Reiten theory.

Nakayama algebras can be grouped into two disjoint classes, namely Nakayama algebras having a directed line as a quiver and Nakayama algebras having an oriented cycle as a quiver.
We call Nakayama algebras with a line as a quiver LNakayama algebras and those with a cycle as a quiver we call CNakayama algebras.

For Nakayama algebras the problem to classify the higher Auslander algebras among them is open, and in general there are many Nakayama algebras with $n$ simple modules that are higher Auslander algebras of fixed global dimension $r$ for some $r \geq 2$. A classification of higher Auslander algebras is known for truncated Nakayama algebras, see \cite{CIM} and \cite{V}, but besides those results it seems not much is known.

In this article we give a classification of Nakayama algebras with $n$ simple modules that are higher Auslander algebras of global dimension at least $n$. Surprisingly, there is exactly one such algebra with $n$ simple modules and global dimension $r$ for $n \leq r \leq 2n-2$, and the class of those algebras coincides with other naturally appearing classes of Nakayama algebras. We recall that the global dimension of a Nakayama algebra with finite global dimension is bounded above by $2n-2$ by \cite{Gus}.

We define the \emph{defect} Def($A$) of a finite dimensional algebra $A$ as the number of indecomposable injective non-projective $A$-modules. Note that the defect measures how far away an algebra is from being selfinjective, since the selfinjective algebras are exactly those with defect $0$. We define the class of \emph{defect 1 Nakayama algebras}, short D1-Nakayama algebras, as the Nakayama algebras having defect $1$.
We define the \emph{super dominant dimension} $\sdomdim(A)$ of an algebra $A$ as the sum of all dominant dimensions of the indecomposable projective non-injective $A$-modules.
Our main tool for the classification of Nakayama algebras of large global dimension that are higher Auslander algebras is the following inequality, that is our first main theorem and improves inequalities obtained in \cite{Mar} for Nakayama algebras.

\begin{theorem*}
Let $A$ be a finite dimensional non-selfinjective Nakayama algebra with $n$ simple modules.
Then
$$\sdomdim(A) \leq 2n-\Def(A).$$
\end{theorem*}

As a corollary of this theorem, we get that a non-selfinjective Nakayama algebra with $n$ simple modules and dominant dimension at least $n$ has to be a D1-Nakayama algebra.
Thus we are directly lead to look at the classification of higher Auslander algebras among the D1-Nakayama algebras.

\begin{theorem*}
The following are equivalent for a Nakayama algebra with $n$ simple modules.
\begin{enumerate}
\item $A$ is a higher Auslander algebra of global dimension larger than or equal to $n$.
\item $A$ is a D1-CNakayama algebra having finite global dimension.
\item $A$ has dominant dimension at least $n$ and finite global dimension.
\end{enumerate}
Furthermore, for any $n \leq r \leq 2n-2$, there is a unique Nakayama algebra $A$ with global dimension $r$ that is a higher Auslander algebra.
\end{theorem*}

We also give formulas for the Kupisch series of those CNakayama algebras that are higher Auslander algebras of global dimension larger than or equal to $n$.
The previous theorem motivates in particular the conjecture that for a given $n$ and $r$ with $2 \leq r \leq 2n-2$, there exists a higher Auslander algebra among the Nakayama algebras with $n$ simple modules and global dimension $r$. The theorem confirms the conjecture for half of the values for $r$, we briefly discuss the conjecture at the end of this article.

For $n \geq 2$ and $1 \leq m \leq n-1$, we define the algebras $\mathcal{Z}_{n,m}$ as the unique CNakayama algebra with $n$ simple modules and global dimension $n+m-1$ that is a higher Auslander algebra as in the previous theorem.

By \cite{Mad}, a CNakayama algebra has finite global dimension if and only if there is a simple module of even projective dimension.
In \cite{MM} we proved that a Nakayama algebra with $n$ simple modules and having a simple module of even projective dimension has global dimension bounded by $n+m-1$ when $m \geq 1$ is chosen minimal such that there is a simple module of projective dimension $2m$.
It was left open whether the bound $n+m-1$ is attained for each $1 \leq m \leq n-1$, and in this article we show that this bound is indeed uniquely attained for those values.

\begin{theorem*}
Let $A$ be a CNakayama algebra with $n$ simple modules and finite global dimension.
Choose $m$ minimal such that there is a simple $A$-module with projective dimension $2m$ and assume $A$ has global dimension $n+m-1$. Then $A$ is isomorphic to $\mathcal{Z}_{n,m}$. Furthermore, for any $n \geq 2$ and $1 \leq m \leq n-1$, the minimal projective dimension among the simple $\mathcal{Z}_{n,m}$-modules with even projective dimension is $2m$.
\end{theorem*}

This theorem improves the inequality proved in \cite{MM} and applies for example to give a new improved inequality for the global dimension of quasi-hereditary Nakayama algebras. The second named author is thankful to Emre Sen for useful discussions. Several results were suggested by experiments with the GAP-package QPA, see \cite{QPA}.

\section{Preliminaries}

Assume all algebras are finite dimensional over a field $K$, connected and non-semisimple if nothing is stated otherwise. All modules are assumed to be finitely generated.
We assume the reader is familiar with the basics of representation theory and homological algebra of finite dimensional algebras and especially Nakayama algebras.
We refer for example to the textbooks \cite{AnFul}, \cite{ARS} and \cite{SkoYam} that treat the necessary basics and have chapters on Nakayama algebras.
Note that Nakayama algebras are also often called serial algebras (for example in the book \cite{AnFul}).
For simplicity we assume that our Nakayama algebras are given by quiver and relations and their vertices are labelled from $0$ to $n-1$ when the algebra has $n$ simple modules. Thus we can view the vertices for CNakayama algebras as elements in $\mathbb{Z}/\mathbb{Z}n$ and the quiver of a CNakayama algebra is an oriented cycle. Also in the LNakayama case we sometimes do computations with vertex labels modulo $n$.
Note that assuming that our algebras are given by quiver and relations is no loss of generality in case the field $K$ is algebraically closed since all algebras are in this case Morita equivalent to a quiver algebra and all our notions are invariant under Morita equivalence. For an algebra $A$, let $D \coloneqq \Hom_K(-,K)$ denote the natural duality of $A$, and let $J$ denote the Jacobson radical of $A$.
Recall that Nakayama algebras can be defined as the quiver algebras having one of the following quivers.
The quiver of a CNakayama algebra:
$$Q_{C_n}=\begin{xymatrix}{ &  \circ^0 \ar[r] & \circ^1 \ar[dr] &   \\
\circ^{n-1} \ar[ur] &     &     & \circ^2 \ar[d] \\
\circ^{n-2} \ar[u] &  &  & \circ^3 \ar[dl] \\
   & \circ^5 \ar @{} [ul] |{\ddots} & \circ^4 \ar[l] &  }\end{xymatrix}$$
\newline
\newline
The quiver of an LNakayama algebra:
$$Q_{L_n}=\begin{xymatrix}{ \circ^0 \ar[r] & \circ^1 \ar[r] & \circ^2 \ar @{} [r] |{\cdots} & \circ^{n-2} \ar[r] & \circ^{n-1}}\end{xymatrix}$$

The \emph{Kupisch series} of a Nakayama algebra with $n$ simple modules is by definition the list $[c_0,c_1,\dots,c_{n-1}]$ in this order, where $c_i$ is the vector space dimension of the indecomposable projective module $e_i A$. Note that a Nakayama algebra has a line as a quiver if and only if $c_{n-1}=1$ and that the Kupisch series is only determined up to cyclic shifts for CNakayama algebras. The Kupisch series uniquely determines the algebra and a Nakayama algebra is selfinjective if and only if the Kupisch series is constant.

The \emph{dominant dimension} $\domdim(M)$ of a module $M$ with minimal injective coresolution $$0 \to M \to I_0 \to I_1 \to \dots \to I_i \to \dots$$ is defined as zero in case $I_0$ is not projective and equal to $$\sup \{ n \mid I_i \text{ is projective for } i=0,1,\dots,n \}+1$$ in case $I_0$ is projective.
The codominant dimension $\codomdim(M)$ of a module is defined as the dominant dimension of $D(M)$ as an $A^{\op}$-module.
The dominant dimension of an algebra $A$ is defined as the dominant dimension of the regular module and the codominant dimension of $A$ as the codominant dimension of $D(A)$. It can be shown that the dominant and codominant dimension of $A$ coincide, see for example section 3 in \cite{Yam}. An algebra is called a \emph{QF-3 algebra} in case it has dominant dimension at least one. Nakayama algebras are always QF-3 algebras, see for example \cite[Theorem 32.2]{AnFul}. An algebra is called a \emph{Gorenstein algebra} in case the injective dimensions of the regular module as a right and left module coincide and are finite. The common value is then called the \emph{Gorenstein dimension}. In case the algebra has finite global dimension, the Gorenstein dimension is finite and coincides with the global dimension. The \emph{finitistic dimension} of an algebra is defined as the supremum of all projective dimensions of modules having finite projective dimension. In case the Gorenstein dimension is finite, the finitistic dimension coincides with the Gorenstein dimension, see for example \cite[Lemma 2.3.2]{Che}. The global dimension, Gorenstein dimension, dominant dimension, and finitistic dimension of an algebra $A$ are denoted by $\gldim(A)$, $\Gdim(A)$, $\domdim(A)$, and $\findim(A)$ respectively.

For $M$ a module and $i>0$, we let $\Omega^i(M)$ denote the $i$th syzygy in a minimal projective resolution of $M$. Similarly $\Omega^{-i}(M)$ denotes the $i$th cosyzygy in a minimal injective coresolution of $M$. A special feature of Nakayama algebras is that syzygies and cosyzigies of an indecomposable module are indecomposable or zero. The projective dimension and injective dimension of $M$ are denoted by $\pd(M)$ and $\id(M)$ respectively. We let $\soc(M)$ denote the socle of $M$ and $\top(M)$ denote the top of $M$.

Following \cite{Rin} we define the \emph{resolution quiver} of a CNakayama algebra $A$ with $n$ simple modules as the quiver with vertices $i=0,\dots,n-1$ (corresponding to the simple $A$-modules $S_i$) and there is an arrow from $i$ to $j$ if and only if $\tau(\soc(P(S_i)))=S_j$, where $\tau$ is the Auslander-Reiten translation and $P(S_i)$ denotes the projective cover of the simple module $S_i$. Note that the condition $\tau(\soc(P(S_i)))=S_j$ is equivalent to $j=i+c_i$ (modulo $n$). The resolution quiver has no sinks and each connected component contains a unique (minimal) cycle. The \emph{injective resolution quiver} of $A$ is defined to be the resolution quiver of the opposite algebra of $A$.

A point $i$ in the resolution quiver is called \emph{black} in case $S_i$ has projective dimension at least two and it is called \emph{red} otherwise. A cycle in this resolution quiver is called \emph{black} in case every simple module in the cycle has projective dimension at least two. A black point in a black cycle is called \emph{cyclically black}.
Following \cite{S}, the \emph{weight $w(C)$ of a cycle} $C$ in the resolution quiver of $A$ is defined as $w(C):= \sum\limits_{i=1}^{l}{\frac{c_{r_i}}{n}}$, when the cycle contains the modules $S_{r_i}$ for $i=1,\dots,l$ and the $c_i$ denote the dimensions of the projective covers of $S_i$. It is known that the weight of each cycle is the same, see for example \cite{S}. We thus speak of the weight of the resolution quiver of $A$.

\begin{theorem} \label{Shentheorem}
Let $A$ be a CNakayama algebra.
\begin{enumerate}
\item $A$ has finite global dimension if and only if the resolution quiver of $A$ is connected and the weight of the resolution quiver is 1.
\item $A$ has finite Gorenstein dimension if and only if every cycle in the resolution quiver of $A$ is black.
\end{enumerate}
\end{theorem}
\begin{proof}
See \cite{S}.
\end{proof}

Recall that by the Morita-Tachikawa correspondence an algebra $A$ has dominant dimension at least two if and only if $A \cong \End_B(M)$ for an algebra $B$ with a generator-cogenerator $M$ of $\mod B$. In this case $B$ is uniquely determined by $A$. We call $B$ the \emph{base} algebra of $A$. In case $A$ has dominant dimension at least two and selfinjective base algebra, $A$ is called a \emph{Morita algebra} following \cite{KerYam}.
In case $M$ is a generator-cogenerator such that $A \cong \End_B(M)$ has finite global dimension equal to the dominant dimension, $A$ is called a \emph{higher Auslander algebra} and $M$ a \emph{cluster tilting object}, following \cite{Iya}.
In case $M$ is a generator-cogenerator such that $A \cong \End_B(M)$ has finite Gorenstein dimension equal to the dominant dimension, $A$ is called a \emph{minimal Auslander-Gorenstein algebra} and $M$ a \emph{precluster tilting object}, following \cite{IyaSol}.

If $x$ is a real number, we use the notation $[x]$ to denote the greatest integer less than or equal to $x$ and $[[x]]$ to denote the least integer greater than or equal to $x$. We sometimes use the notation $\equiv_n$ to denote congruence modulo $n$.

\section{Bounds for the dominant dimension of Nakayama algebras}

Recall our definition of the super dominant dimension of an algebra as the sum of all dominant dimensions of the indecomposable projective non-injective modules. Dually, the \emph{super codominant dimension} $\scodomdim(A)$ of an algebra $A$ is defined as the sum of all codominant dimension of the indecomposable injective non-projective $A$-modules. Note that $\scodomdim(A)=\sdomdim(A^{op})$.

\begin{theorem} \label{inequality}
Let $A$ be a non-selfinjective Nakayama algebra with $n$ simple modules.
Then
$$\sdomdim(A) \leq 2n-\Def(A).$$
\end{theorem}

\begin{proof}
As we prefer to deal with projective resolutions rather than injective resolution we prove the dual inequality $\scodomdim(A) \leq 2n-\Def(A)$, which implies the theorem since $\Def(A)=\Def(A^{\op})$ and the opposite algebra of a Nakayama algebra is again a Nakayama algebra.
First note that every indecomposable injective non-projective $A$-module $L$ has finite codominant dimension, since otherwise the modules $\Omega^{i}(L)$ would have injective dimension equal to $i$, and thus the algebra $A^{\op}$ would have infinite finitistic dimension, which is absurd since $A^{\op}$ is representation-finite.

To every indecomposable projective module $P=e_i A$ we associate a set of indecomposable modules ${}_{P}{\Int}$, that we call the left interval (associated to $P$), as follows.
$$ {}_{P}{\Int}:= \{ M \in \mod A \mid \top(M)=\top(P) \}= \{ M = e_i A/ e_i J^k \mid k=1,\dots,c_i \}.$$ For Nakayama algebras every indecomposable injective module $I$ is a quotient module of an indecomposable projective module: $I=e_s A/ e_s J^t$. We associate to $I$ the set of indecomposable modules $$\Int_I:= \{ M \in \mod A \mid \soc(M)=\soc(I) \} = \{ M = e_s J^r /e_s J^t \mid r=0,\dots,t-1 \},$$ called the right interval (associated to $I$).
Note that $e_s J^r /e_s J^t \cong e_{s+r} A / e_{s+r} J^{t-r}$, and thus $\Int_I= \{ M = e_{s+r} A / e_{s+r} J^{t-r} \mid r=0,\dots,t-1 \}$.

We now define a map $\phi$ which takes as argument either a right interval $\Int_I$ associated to an indecomposable injective module $I$ or a left interval ${}_{P}{\Int}$ associated to an indecomposable projective-injective module $P$. The value of $\phi$ is a left interval or a right interval respectively.

For $I=e_s A/ e_s J^t$ indecomposable injective we define $$\phi(\Int_I):= {}_{P_{(I)}}{\Int},$$ where $P_{(I)}$ is the indecomposable projective module with $\top(P_{(I)})=\top(\Omega^1(\soc(I)))$ whenever $\soc(I)$ is not projective. In case $\soc(I)$ is projective, which can only happen for LNakayama algebras, we let $P_{(I)}$ be the unique indecomposable projective module with $\top(P_{(I)})$ injective.
Note that $\soc I \cong e_{s+t-1} A / e_{s+t-1} J$ and therefore whenever $\soc(I)$ is not projective we have $\top(\Omega^1(\soc(I)))$ is the simple module corresponding to vertex $s+t$. Thus $P_{(I)}=e_{s+t}A$. The same holds true when $\soc(I)$ is projective.
Note furthermore that for every non-projective module $M= e_{s+r} A / e_{s+r} J^{t-r} \in \Int_I$, we have $\top(\Omega^1(M))=\top(P_{(I)})$, and thus $\Omega^1(M) \in {}_{P_{(I)}}{\Int}$. When $\soc(I)$ is projective, every module in $\Int_I$ is projective.

Now for an indecomposable projective-injective module $P=e_i A$, we define $$\phi({}_{P}{\Int})\coloneqq \Int_P,$$ when we view $P$ also as an injective module.
We have for every non-projective $M=e_i A/e_i J^k \in {}_{P}{\Int}$ that $\soc(\Omega^1(M))=\soc(e_i J^k)=\soc(P)$, and thus $\Omega^1(M) \in \Int_P$.

By dual considerations we can define a map $\phi'$ which takes as argument either a left interval associated to an indecomposable projective module $P$ or a right interval ${\Int_I}$ associated to an indecomposable projective-injective module $I$. If $P$ is indecomposable projective, we define $\phi'({}_{P}{\Int}):=\Int_{I_{(P)}}$, where $I_{(P)}$ is the indecomposable injective module with $\soc(I_{(P)})=\soc(\Omega^{-1}(\top(P)))$ whenever $\top(P)$ is not injective. In case $\top(P)$ is injective, which can only happen for LNakayama algebras, we let $I_{(P)}$ be the unique indecomposable injective module with $\soc(I_{(P)})$ projective. If $I$ is indecomposable projective-injective, we define $\phi'(\Int_I):={}_{I}{\Int}$.

Let $E$ be a left or right interval. Since $\soc(\Omega^{-1}(\top(\Omega^1(\soc(I))))) \cong \soc(I)$ for each indecomposable injective module $I$ with $\soc(I)$ not projective, and $\top(\Omega^1(\soc(\Omega^{-1}(\top(P)))) \cong \top (P)$ for each indecomposable projective module $P$ with $\top(P)$ not injective, and $\soc(I)$ is projective if and only if $\top(P_{(I)})$ is injective, we have $\phi'(\phi(E))=E$ whenever $\phi(E)$ is defined and $\phi(\phi'(E))=E$ whenever $\phi'(E)$ is defined.

We can now define a bijection $\pi$ between the set of indecomposable injective non-projective modules $\mathcal{L}$ and the set of indecomposable projective non-injective modules $\mathcal{R}$.
For $I \in \mathcal{L}$ we repeatedly apply $\phi$ to $\Int_I$ until we obtain an interval ${}_{P}{\Int}$ with $P$ indecomposable projective non-injective.
Then we set $\pi(I):=P$ and this defines a bijection since we can use $\phi'$ to recover $I$ from $P$.

We can label such a chain of intervals as $L_i \to \pi(L_i)$ for $L_i \in \mathcal{L}$. The intervals in this chain are $$\Int _{L_i},\phi(\Int_{L_i}),\phi^2(\Int_{L_i}),\dots,{}_{\pi(L_i)}{\Int}.$$ The intervals in this chain never repeat since $\phi'(\phi(E))=E$ whenever $\phi(E)$ is defined for an interval $E$ and $\Int _{L_i}$ is not in the image of $\phi$.
Since when a non-projective module $M$ appears in an interval $E$ then $\Omega(M)$ appears in $\phi(E)$, and the projective module $\pi(L_i)$ is not injective, the codominant dimension of $L_i$ is bounded above by one less than the number of intervals in the chain. If there are $r_i$ projective-injective modules appearing in this chain, then there are $2r_i+2$ intervals, so we can say $$\codomdim(L_i)\le 2r_i+1.$$

We show that an indecomposable projective-injective module $P$ can appear in at most one such chain of intervals, and if it does we call it \emph{bounded} and denote the number of bounded projective-injective modules by $b(A)$. Indeed, if one starts from a bounded projective-injective module and the left and right intervals associated to it, and constructs a chain by repeatedly applying $\phi'$ to the left and $\phi$ to the right, the construction will end when one obtains some interval $\Int _{L_i}$ with $L_i \in \mathcal{L}$ on the left and the corresponding interval ${}_{\pi(L_i)}{\Int}$ on the right, and thus one has constructed the unique $L_i \to \pi(L_i)$ chain containing this bounded projective-injective module.

We get
$$\scodomdim(A)=\sum_{i=1}^{\Def(A)} \codomdim(L_i) \leq \sum_{i=1}^{\Def(A)}(2r_i+1)=2b(A)+\Def(A).$$
Now note that $b(A) \leq n-\Def(A)$, and thus
$$\scodomdim(A) \leq 2b(A)+\Def(A) \leq 2(n-\Def(A))+\Def(A)=2n-\Def(A).$$
\end{proof}

\begin{corollary} \label{inequalitycor}
Let $A$ be a non-selfinjective Nakayama algebra with $n$ simple modules.
\begin{enumerate}
\item $\sdomdim(A) \leq 2n-2$.
\item $\domdim(A) \leq \frac{2n-2}{\Def(A)}.$
\item In case $A$ has dominant dimension $\geq n$, $A$ has $\Def(A)=1$.
\end{enumerate}
\end{corollary}

\begin{proof}
\begin{enumerate}
\item For $\Def(A) \geq 2$, this inequality follows directly from Theorem \ref{inequality}.
For $\Def(A)=1$ the statement reduces to $\domdim(A) \leq 2n-2$, which was proven in \cite[Theorem 2.16]{Mar}.
\item This follows directly from (1), since $\Def(A) \cdot \domdim(A) \leq \sdomdim(A) \leq 2n-2$ gives
$$\domdim(A) \leq \frac{2n-2}{\Def(A)}.$$
\item This is a direct consequence of (2), since in case $\Def(A) \geq 2$ we have
 $$\domdim(A) \leq \frac{2n-2}{\Def(A)} \leq n-1.$$
\end{enumerate}

\end{proof}

Note that the inequality $\sdomdim(A) \leq 2n-2$ is optimal, since the $n$-CNakayama algebra with Kupisch series $[n,n+1,n+1,\dots,n+1]$ has super dominant dimension equal to $2n-2$. We note that our inequalities for the dominant dimension of Nakayama algebras obtained in this section are a strong improvement compared to the main result in \cite{Mar}. In \cite{Mar} it was proven that $\domdim(A) \leq 2 (n- \Def(A))$ when $n$ is the number of simple $A$-modules. For example for $\Def(A) \geq 2$ the inequality from \cite{Mar} just gives the bound $\domdim(A) \leq 2n-4$, while our new inequality from Corollary \ref{inequalitycor}(2) gives the much better result $\domdim(A) \leq \frac{2n-2}{2}=n-1$.

We remark that while the dominant dimension of an algebra $A$ coincides with the dominant dimension of the opposite algebra of $A$, this is no longer true for the super dominant dimension as we show in the next example.

\begin{example}
The Nakayama algebra with Kupisch series $[2, 4, 3, 3, 3]$ has super dominant dimension equal to $5$ but the super dominant dimension of the opposite algebra is equal to $4$.
\end{example}

\begin{question}
Is there a connected finite dimensional non-selfinjective algebra $A$ with $n$ simple modules and $\domdim(A) > 2n- \Def(A)$?
\end{question}

The question is motivated by a conjecture of Yamagata (see \cite{Yam}) that states that the dominant dimension of a general non-selfinjective algebra $A$ is bounded a function $f$ depending only on the number of simple modules of the algebra. To our knowledge no finite dimensional non-selfinjective algebra with $n$ simple modules and $\domdim(A)>2n-1$ is known, so one might wonder whether one can take $f(n)=2n$ for the function in Yamagata's conjecture.

\section{D1-Nakayama algebras with finite global dimension}\label{sectd1}

We call Nakayama algebras with a unique indecomposable projective non-injective module \emph{D1-Nakayamas algebras} since they are exactly those with defect $1$.
Note that the unique D1-LNakayama algebra for a given number of simple modules $n$ is the one with Kupisch series $[2,2,\dots,2,1]$, and it is easily seen that this algebra has global dimension $n-1$ and dominant dimension $n-1$.

The CNakayama algebras with exactly one projective non-injective indecomposable module are those with Kupisch series of the form $[a,a,\dots,a,a+1,a+1,\dots,a+1]$, where $a \geq 2$. This follows directly from the fact that an indecomposable projective module $e_i A$ is injective if and only if $c_{i-1} \leq c_i$, see for example \cite[Theorem 32.6]{AnFul}. We denote such a CNakayama algebra with this Kupisch series as $N_{n,a,s}$ when the algebra has $n \geq 2$ simple modules, the number $a$ is as in the Kupisch series and $s<n$ counts how often $a$ appears.
For example $[2,2,2,3,3]$ is denoted by $N_{5,2,3}$.

We use the following theorem, which is a special case of results in section 5 of \cite{Mar2}.

\begin{theorem} \label{findimtheorem}
Let $A$ be a representation-finite algebra with exactly one projective non-injective indecomposable module and dominant dimension at least one. Then the dominant dimension of $A$ equals the finitistic dimension.
\end{theorem}

This theorem applies to all our Nakayama algebras with exactly one projective non-injective indecomposable module. The theorem implies that D1-Nakayama algebras have finite global dimension if and only if their global dimension coincides with the dominant dimension and they have finite Gorenstein dimension if and only if their Gorenstein dimension coincides with the dominant dimension.

We state some basic results about the D1-CNakayama algebras.
We need the following lemma that allows us to restrict to the algebras $N_{n,a,s}$ for $a \leq n-1$ when classifying such algebras with finite global dimension.

\begin{lemma} \label{gldiminfinitecrit}
Let $A$ be a CNakayama algebra with $n$ simple modules.
\begin{enumerate}
\item In case $c_i \geq n+1$ for the Kupisch series $[c_0,c_1,\dots, c_{n-1}]$ for all $i$, the algebra $A$ has infinite global dimension.
\item In case $a=n$ and $A=N_{n,a,s}$, the algebra $A$ has finite global dimension if and only if $A$ has Kupisch series $[n,n+1,n+1,\dots,n+1]$. In this case the global dimension of $A$ is $2n-2$.
\end{enumerate}
\end{lemma}

\begin{proof}
\begin{enumerate}
\item This follows immediately from Theorem \ref{Shentheorem}(1) since the weight of the resolution quiver can not be equal to one when $c_i \geq n+1$ for all $i$.
\item Each entry $n$ in the Kupisch series constitutes a cycle in the resolution quiver. By Theorem \ref{Shentheorem}(1), finite global dimension implies a connected resolution quiver, so in case $a=n$ and $A=N_{n,a,s}$ finite global dimension implies $s=1$. The Nakayama algebra with Kupisch series $[n,n+1,n+1,\dots,n+1]$ has global dimension equal to $2n-2$, see for instance \cite{Gus}.
\end{enumerate}
\end{proof}

\begin{proposition} \label{basicsabout1nak}
Let $A=N_{n,a,s}$ be a D1-CNakayama algebra.
\begin{enumerate}
\item The unique projective non-injective indecomposable module is $e_0 A$ and the unique injective non-projective indecomposable module is $e_s A/e_s J^a$.
\item The dominant dimension and the finitistic dimension of $A$ equal the codominant dimension of the module $e_s A/e_s J^a$.
\item The dominant dimension of $A$ is equal to one if and only if $s+a \equiv 0$ (modulo $n$), and in this case the algebra has infinite Gorenstein dimension.
\item The dominant dimension of $A$ is equal to two if and only if $s+a+1 \equiv 0$ (modulo $n$). In this case $A$ is Gorenstein, and $A$ has finite global dimension if and only if $n=2$ and $A$ has Kupisch series $[2,3]$.
\end{enumerate}
\end{proposition}

\begin{proof}
\begin{enumerate}
\item Since $c_0 < c_{n-1}$, the indecomposable projective module $e_0 A$ is not injective. Since $c_{s-1} < c_s$, the indecomposable module $e_s A/e_s J^a$ is injective with projective cover $e_s A \to e_s A/e_s J^a \to 0$.
\item This follows from (1) and Theorem \ref{findimtheorem} together with the fact that the dominant dimension of an algebra equals its codominant dimension.
\item The dominant dimension of $A$ equals the codominant dimension of the unique indecomposable injective non-projective module $T \coloneqq e_s A/e_s J^a$ by (2). The minimal projective resolution of $e_s A/e_s J^a$ starts as follows:
$$\dots \to e_{s+a}A \to e_s A \to e_s A/e_s J^a \to 0.$$
Since $A$ is a Nakayama algebra, it has dominant dimension at least one. Thus the codominant dimension equals one if and only if $e_{a+s}A$ is non-injective if and only if $e_{a+s}A=e_0A$ if and only if $0 \equiv a+s$ (modulo $n$). In this case one has $\Omega^{1}(T)=e_s J^a$ and $\Omega^{2}(T)=e_{0} J$, which is not projective. This means that the algebra is not Gorenstein, since when the Gorenstein dimension is finite it coincides with the finitistic dimension.
\item The minimal projective resolution of $e_s A/e_s J^a$ continues as follows.
$$\dots \to e_{s+a+1}A \to e_{s+a}A \to e_s A \to e_s A/e_s J^a \to 0.$$
Thus $A$ has dominant dimension equal to two if and only if $s+a+1 \equiv 0$ (modulo $n$), since $s+a+1 \equiv 0$ (modulo $n$) implies $s+a \neq 0$ (modulo $n$).

Suppose the dominant dimension of $A$ is equal to two. Then the projective resolution above becomes
$$0 \to e_0A \to e_{n-1}A \to e_s A \to e_s A/e_s J^a \to 0.$$ Therefore the Gorenstein dimension of $A$, which equals the projective dimension of $e_s A/e_s J^a$, is equal to two. Suppose there is $0 \leq b <n$ such that $e_b A$ and $e_{b+1}A$ have the same vector space dimension. Then $\Omega^2(e_b A/e_b J)=\soc(e_{b+1}A)$ which is simple non-projective, and as a consequence the global dimension of $A$ is infinite since it is larger than the dominant dimension. So if $A$ has finite global dimension we must have $n=2$, and in this case we must have $a=2$ by Lemma \ref{gldiminfinitecrit}. The algebra with Kupisch series $[2,3]$ has global dimension equal to two.
\end{enumerate}
\end{proof}

\begin{lemma}\label{OmegaTwo}
Let $A=N_{n,a,s}$ be a D1-CNakayama algebra. Then $U=S_{n-1}$ is the unique simple $A$-module $U$ with $\Omega^2(U)=0$, and the unique simple $A$-module $V$ such that $\Omega^2(V)$ has length two is given by $V=S_{s-1}$. For all other simple $A$-modules $S$ we have that $\Omega^2(S)$ is a simple module.
\end{lemma}

\begin{proof}
Assume $A=N_{n,a,s}$.
Note that no simple module of a CNakayama algebra is projective and thus $\Omega^2(U)=0$ for a simple module $U$ is equivalent to $U$ having projective dimension equal to one.
A simple module $S_i$ has projective dimension equal to one if and only if $e_i J \cong e_{i+1}A$, which is equivalent to $c_i -1= c_{i+1}$ by comparing dimensions.
But by the form of the Kupisch series of a $A$, this is true if and only if $i=n-1$.
Now assume $S_i$ does have projective dimension at least two. Then $\Omega^2(S_i)=e_{i+1}J^{c_i-1}$ and this module has vector space dimension equal to $c_{i+1}-c_i+1$. Thus the vector space dimension of $\Omega^2(S_i)$ is equal to two if and only if $c_{i+1}-c_i=1$, which is equivalent to $i=s-1$.
Now for $i \neq n-1$ and $i \neq s-1$, we have that $\dim_K(\Omega^2(S_i))=c_{i+1}-c_i+1=1$ and thus $\Omega^2(S_i)$ is simple.
\end{proof}

Lemma \ref{gldiminfinitecrit} allows us to focus the case $a \leq n-1$ for the classification of D1-CNakayama algebras $N_{n,a,s}$ of finite global dimension in the following.
By Lemma \ref{OmegaTwo}, for a D1-CNakayama algebra $A$ with $n$ simple modules the simple module $S_{n-1}$ is the unique simple module with projective dimension equal to one. We define the idempotent $$e \coloneqq e_0 + e_1 + \dots + e_{n-2},$$ and fix this notation for  D1-CNakayama algebras in this section. Note that for any idempotent $f$ in a Nakayama algebra $A$, also $fAf$ is a Nakayama algebra, see for example \cite[Proposition 2.19]{Mar}.
We will look at the Nakayama algebra $eAe$ for a given D1-CNakayama algebra $A$ in the following.

\begin{lemma} \label{formulaforeAe}
Let $A=N_{n,a,s}$ with $a \leq n-1$ and let $[c_0,c_1,\dots,c_{n-1}]$ be the Kupisch series of $A$. Then the algebra $eAe$ has Kupisch series $[r_0,r_1,\dots,r_{n-2}]$ with $r_i=c_i- \phi(i)$, where for $0 \leq i \leq n-2$,
\begin{equation*}
    \phi(i) =
    \begin{cases*}
      1 & if $c_i \geq n-i$,\\
      0 & otherwise.
    \end{cases*}
  \end{equation*}
\end{lemma}

\begin{proof}
We know that $eAe$ is again a Nakayama algebra and thus it is uniquely determined by its Kupisch series.
The indecomposable projective $eAe$-modules are given by $e_i A e$ for $0 \leq i \leq n-2$, and thus it is enough to calculate their vector space dimensions.
We have $r_i \coloneqq \dim_K(e_i A e)= \dim_K(e_i A) - \dim_K(e_i A e_{n-1})=c_i-\dim_K(e_i A e_{n-1})$. Now clearly $\phi(i) \coloneqq \dim_K(e_i A e_{n-1})$ is equal to one if and only if $c_i \geq n-i$ and equal to zero otherwise.
\end{proof}

We will also need the following proposition.

\begin{proposition} \label{propogldimreduce}
Let $A$ be a CNakayama algebra and suppose $f$ is a primitive idempotent such that $fA/fJ$ has projective dimension equal to one. Let $e=1-f$.
Then $A$ has finite global dimension if and only if $eAe$ has finite global dimension. In case $A$ has finite global dimension, we have $$\gldim(A) \leq \gldim(eAe) +2.$$
\end{proposition}

\begin{proof}
See \cite[Lemma 4]{BFVZ}.
\end{proof}

In the case of D1-Nakayama algebras, the previous proposition can be improved as we show next.

\begin{lemma} \label{gldimequalitylemma}
Let $A$ be a D1-CNakayama algebra of finite global dimension. Then $$\gldim(A)=\gldim(eAe)+2.$$
\end{lemma}

\begin{proof}
For Nakayama algebras, any finite minimal projective resolution of an indecomposable non-projective module has to end with an inclusion between indecomposable projectives.
In the D1-CNakayama case, there is only one possible such inclusion between indecomposable projective modules, namely $0 \rightarrow e_{n-1} J \rightarrow e_{n-1}A$, because $e_{n-1} J \cong e_0 A$ is the unique indecomposable projective non-injective $A$-module and any finite minimal projective resolution of an indecomposable non-injective module ends with a non-injective projective module.

If we apply the exact functor $F=\Hom(eA, -)$ to the exact sequence $0 \to e_{n-1} J \to e_{n-1}A \to S_{n-1} \to 0$, then we get an isomorphism $0 \rightarrow F(e_{n-1}J) \rightarrow F(e_{n-1}A) \rightarrow 0$ since $F(S_{n-1})=0$.  Since $F(e_{n-1}J) \cong F(e_0A) \cong e_0 A e$ is projective as an $eAe$-module, also $F(e_{n-1}A)$ is projective as an $eAe$-module.
Note that for $0 \leq i \leq n-2$, applying the functor $F$ to the exact sequence
$$0 \to e_i J \to e_i A \to S_i \to 0$$
gives the exact sequence
$$0 \to e_i Je \to e_i Ae \to S_i e \to 0.$$
We have that $S_i e$ is non-zero and simple, since $S_i$ and thus also $S_i e$ has vector space dimension equal to one. For $i \neq j$, the modules $S_i e$ and $S_j e$ are non-isomorphic since their projective covers $e_iA e$ and $e_j A e$ are non-isomorphic.
Thus every simple $eAe$-module is isomorphic to $F(S_i)$ for some with some $0 \leq i \leq n-2$.
If we apply $F$ to the minimal projective $A$-resolution of a simple $A$-module $S$, then we get a resolution of $F(S)$ as an $eAe$-module. By the arguments above, this resolution is projective and it can be made two steps shorter since the A-resolution ends with  $0 \to e_{n-1}J \to e_{n-1}A$, which becomes an isomorphism after applying $F$. Since this applies to all simple modules $S$, we conclude $\gldim(A) \geq \gldim(eAe)+2$.
By Proposition \ref{propogldimreduce}, we also have $\gldim(A) \leq \gldim(eAe)+2$, which finishes the proof.
\end{proof}

\begin{theorem} \label{uniqueness proof}
For fixed $n \geq 2$ and $r$ with $n \leq r \leq 2n-2$, there exists a unique D1-CNakayama algebra with $n$ simple modules and finite global dimension $r$. These are the only D1-CNakayama algebras of finite global dimension.
\end{theorem}

\begin{proof}
The theorem is easy to verify for $n=2$ where the only D1-CNakayama algebra with finite global dimension is the one with Kupisch series [2,3]. For $n \geq 3$, by Proposition \ref{basicsabout1nak} all D1-CNakayama algebras with finite global dimension have global dimension strictly larger than $2$. When $a=n$ there is a unique D1-CNakayama algebras with finite global dimension $2n-2$ by Lemma \ref{gldiminfinitecrit}(2). In case $a>n$, a D1-CNakayama algebra can never have finite global dimension by Lemma \ref{gldiminfinitecrit}(1).

Assume $n \geq 3$ and $a \leq n-1$. For the rest of the proof let $C_n$ denote the Nakayama algebra with Kupisch series $[2,2,2,\dots,2,1]$ and $n$ simple modules. As previously noted we have $\gldim (C_n)=n-1$.
For a fixed $n \geq 3$ let $\mathcal{E} \colon X_1 \rightarrow X_2$ be the map sending a D1-CNakayama algebra $A$ to the Nakayama algebra $eAe$, where
$$X_1 \coloneqq \{ N_{n,a,s} \mid 2 \leq a \leq n-1 \text{ and } \domdim(N_{n,a,s}) \geq 3 \}$$ and
$$X_2 \coloneqq \{ N_{n-1,a,s} \mid 2 \leq a \leq n-2 \} \cup \{C_{n-1} \}.$$
Note that by Proposition \ref{basicsabout1nak} we have $\domdim(N_{n,a,s}) \geq 3$ if and only if $a+s \neq 0$ (modulo $n$) and $a+s+1 \neq 0$ (modulo $n$). Since $2 \leq a \leq n-1$ and $1 \leq s \leq n-1$, we have $s+a \leq 2n-2$ and $s+a+1 \leq 2n-1$. So $s+a=0$ (modulo $n$) if and only if $s+a=n$, and $s+a+1=0$ (modulo $n$) if and only if $s+a=n-1$.
This gives us that $N_{n,a,s} \in X_1$ if and only if $s \neq n-a$ and $s \neq n-1-a$. Computing the cardinality of the sets $X_1$ and $X_2$ we find that both have $(n-3)(n-2)+1$ elements.

Now by explicit calculations using the formula for the Kupisch series in Lemma \ref{formulaforeAe}, we have
\begin{enumerate}
\item $\mathcal{E}(N_{n,2,n-1})=C_{n-1}$.
\item In case $s>n-a$ and $a>2$ we have $\mathcal{E}(N_{n,a,s})=N_{n-1,a-1,s-(n-a)}$.
\item In case $s<n-a-1$ we have $\mathcal{E}(N_{n,a,s})=N_{n-1,a,s+a}.$
\end{enumerate}
In case (3) we have $a<n-s-1 \leq n-2$, so $\mathcal{E}$ is well defined as a map $\mathcal{E} \colon X_1 \rightarrow X_2$. Since $(a-1,s-(n-a))=(a',s'+a)$ implies $s=s'+n-1 \geq n$, a contradiction, the map $\mathcal{E}$ is injective.
This proves that $\mathcal{E}$ is a bijection from $X_1$ to $X_2$ since $X_1$ and $X_2$ have the same finite cardinality. We illustrate the map $\mathcal{E}$ for $n=7$ in Table 1, where we write $[n,a,s]$ short for the algebra $N_{n,a,s}$ and canceled entries mean that the entry in question does not have dominant dimension at least three.

\begin{table}[h]
\centering
\begin{tabular}{| c | c c c c c c|}

\hline
$A$ & $[7,2,6]$ & $\cancel{[7,2,5]}$ & $\cancel{[7,2,4]}$ & $[7,2,3]$ & $[7,2,2]$ & $[7,2,1]$ \\
$\mathcal{E}(A)$ & $C_6$ & $-$ & $-$ & $[6,2,5]$ & $[6,2,4]$ & $[6,2,3]$\\
\hline
$A$ & $[7,3,6]$ & $[7,3,5]$ & $\cancel{[7,3,4]}$ & $\cancel{[7,3,3]}$ & $[7,3,2]$ & $[7,3,1]$ \\
$\mathcal{E}(A)$ & $[6,2,2]$ & $[6,2,1]$ & $-$ & $-$ & $[6,3,5]$ & $[6,3,4]$\\
\hline
$A$ & $[7,4,6]$ & $[7,4,5]$ & $[7,4,4]$ & $\cancel{[7,4,3]}$ & $\cancel{[7,4,2]}$ & $[7,4,1]$ \\
$\mathcal{E}(A)$ & $[6,3,3]$ & $[6,3,2]$ & $[6,3,1]$ & $-$ & $-$ & $[6,4,5]$\\
\hline
$A$ & $[7,5,6]$ & $[7,5,5]$ & $[7,5,4]$ & $[7,5,3]$ & $\cancel{[7,5,2]}$ & $\cancel{[7,5,1]}$ \\
$\mathcal{E}(A)$ & $[6,4,4]$ & $[6,4,3]$ & $[6,4,2]$ & $[6,4,1]$ & $-$ & $-$\\
\hline
$A$ & $[7,6,6]$ & $[7,6,5]$ & $[7,6,4]$ & $[7,6,3]$ & $[7,6,2]$ & $\cancel{[7,6,1]}$ \\
$\mathcal{E}(A)$ & $[6,5,5]$ & $[6,5,4]$ & $[6,5,3]$ & $[6,5,2]$ & $[6,5,1]$ & $-$\\
\hline
\end{tabular}
\caption{\textbf{The map $\mathcal{E}$ for $n=7$}}
\label{table:Ordinary_ct_S3}
\end{table}

Now we use induction on the number of simple modules $n$. We assume the statement of the theorem is true for $n-1$ and show that it is then also true for $n$. Let $\mathcal{Z}_{n-1,m}$ for $1 \leq m \leq n-2$ be the unique D1-CNakayama algebra with $n-1$ simple modules and global dimension $n-1+m-1$.
By Lemma \ref{gldiminfinitecrit} and Proposition \ref{basicsabout1nak}, a D1-CNakayama algebra with $n$ simple modules and finite global dimension either belongs to $X_1$ or has Kupisch series $[n,n+1,\dots,n+1]$. Similarly, a D1-CNakayama algebra with $n-1$ simple modules and finite global dimension either belongs to $X_2$ or has Kupisch series $[n-1,n,n,\dots,n]$.
The Nakayama algebra with Kupisch series $[n-1,n,n,\dots,n]$ has global dimension equal to $2n-4$, so it must be isomorphic to $\mathcal{Z}_{n-1,n-2}$ by the uniqueness property in the induction hypothesis, and also by uniqueness no algebras belonging to $X_2$ can have global dimension equal to $2n-4$. Also by the induction hypothesis, it then follows that the subset $Y_2 \subseteq X_2$ consisting of the algebras with finite global dimension is $$Y_2=\{ \mathcal{Z}_{n-1,m} \mid 1 \leq m \leq n-3 \} \cup \{C_{n-1} \}.$$ For $n-2 \leq r \leq 2n-5$ there is a unique algebra in $Y_2$ of global dimension $r$. Since $\mathcal{E}$ is a bijection, it follows from Proposition \ref{propogldimreduce} and Lemma \ref{gldimequalitylemma} that the subset $Y_1 \subseteq X_1$ consisting of the algebras with finite global dimension is equal to the inverse image of $Y_2$ under $\mathcal{E}$, and for $n \leq r \leq 2n-3$ there is a unique algebra in $Y_1$ of global dimension $r$. Since the only other D1-CNakayama algebra with $n$ simple modules and finite global dimension is the one with Kupisch series $[n,n+1,\dots,n+1]$, and this algebra has global dimension equal to $2n-2$, the statement is true for $n$ and the theorem follows by induction.
\end{proof}

The main result of this section can now easily derived from the previous results.

\begin{theorem}
The following are equivalent for a Nakayama algebra $A$ with $n$ simple modules.
\begin{enumerate}
\item $A$ is a higher Auslander algebra of global dimension $r \geq n$.
\item $A$ is a D1-CNakayama algebra having finite global dimension.
\item $A$ has dominant dimension at least $n$ and finite global dimension.
\end{enumerate}
Furthermore, for any $n \leq r \leq 2n-2$, there is a unique Nakayama algebra $A$ with global dimension $r$ that is a higher Auslander algebra.
\end{theorem}

\begin{proof}
Assume (1), so that $A$ is a higher Auslander algebra of global dimension $r \geq n$. By definition $A$ has finite global dimension and dominant dimension $r \geq n$ and thus by Corollary \ref{inequalitycor}(3) it is a D1-Nakayama algebra. Since the global dimension is at least $n$ it must be a CNakayama algebra, which gives (2).

Now assume (2). Then $A$ has global dimension $n+m-1$ for some $m \geq 1$ by Theorem \ref{uniqueness proof} and thus it also has dominant dimension at least $n$ by Theorem \ref{findimtheorem}, which shows (3).

That (3) implies (1) follows again by Corollary \ref{inequalitycor}(3) since having dominant dimension at least $n$ implies that $A$ is a D1-Nakayama algebra, and then $A$ additionally having finite global dimension implies that it is a higher Auslander algebra by Theorem \ref{findimtheorem}.

We saw in Theorem \ref{uniqueness proof} that there is a unique D1-CNakayama algebra with $n$ simple modules and finite global dimension $r$ for $n \leq r \leq 2n-2$.
\end{proof}

We give an examples that shows that in general there is no unique higher Auslander algebra of a given global dimension among Nakayama algebras, and we give another example that shows that $eAe$ is not in general a higher Auslander algebra whenever $A$ is a higher Auslander algebra.

\begin{example}
In general for Nakayama algebras with $n$ simple modules there might be two non-isomorphic algebras that are higher Auslander algebras with global dimension $r$ with $r <n$.
For example for $n=9$ the two algebras with Kupisch series $[ 2, 2, 3, 2, 2, 3, 2, 2, 3 ]$ and $[ 5, 5, 5, 5, 5, 9, 8, 7, 6 ]$ are higher Auslander algebras with global dimension $3$.
\end{example}

\begin{example}
Let $A$ be the CNakayama algebra with Kupisch series $[ 2, 2, 3, 2, 2, 3 ]$.
Then $A$ is a higher Auslander algebra with global dimension $3$.
But the algebra $B=eAe$ is the LNakayama algebra with Kupisch series $[3,2,2,2,1]$.
The algebra $B$ has global dimension $3$ and dominant dimension $1$, and thus it is not a higher Auslander algebra.
\end{example}

We next look at the Kupisch series of the D1-CNakayama algebras that are higher Auslander algebras.
We define $\mathcal{Z}_{n,m}$ for $1 \leq m \leq n-1$ to be the unique D1-CNakayama algebra with $n$ simple modules and global dimension equal to $n+m-1$. Since $\mathcal{E}(N_{n,2,n-1})=C_{n-1}$, we get $\gldim (N_{n,2,n-1})=n$, and therefore $\mathcal{Z}_{n,1}=N_{n,2,n-1}$.
In the proof of Theorem \ref{uniqueness proof} we saw that for $2 \leq m \leq n-2$ the algebras $\mathcal{Z}_{n,m}$ are recursively determined by the condition that $\mathcal{E}(\mathcal{Z}_{n,m})=\mathcal{Z}_{n-1,m-1}$. Since $\mathcal{E}$ is a bijection, we can consider its inverse, which for algebras $N_{n-1,a,s}$ with $2 \leq a \leq n-2$ is given by
\begin{enumerate}
\item In case $s \leq a$ we have $\mathcal{E}^{-1}(N_{n-1,a,s})=N_{n,a+1,s-a+n-1}$.
\item In case $s>a$ we have $\mathcal{E}^{-1}(N_{n-1,a,s})=N_{n,a,s-a}.$
\end{enumerate}
This gives a suitable recursive description of the Kupisch series of $\mathcal{Z}_{n,m}$.
Explicit formulas are given in the following proposition.

\begin{proposition}
Given $n \geq 2$ and $1 \leq m \leq n-1$, then the algebra $\mathcal{Z}_{n,m}$ is the Nakayama algebra $N_{n,a,s}$, where
\begin{enumerate}
  \item if $n-m$ is odd, then $$a=[\frac{2n}{n-m+1}]$$ and $$s= \frac a 2 ((a-1)n-(a+1)m+a-1)+1.$$
  \item if $n-m$ is even, then $$a=[\frac{2(n-1)}{n-m}]$$ and $$s= \frac a 2 ((a-1)n-(a+1)m+2).$$
\end{enumerate}
\end{proposition}

\begin{proof}
As noted before, we have $\mathcal{Z}_{n,n-1}=N_{n,n,1}$, which satisfies the above formulas for $n-m$ odd. For $\mathcal{Z}_{n,1}$ the formulas both in the odd and the even case give $\mathcal{Z}_{n,1}=N_{n,2,n-1}$, a fact that we already have established. This will serve as a basis for an induction proof. For the induction step we use that $\mathcal{E}^{-1}(\mathcal{Z}_{n-1,m-1})=\mathcal{Z}_{n,m}$ and the formulas for $\mathcal{E}^{-1}$ given above. We split into odd and even cases.
\begin{enumerate}
\item Let $2 \leq m \leq n-2$ and assume $n-m$ is odd. Assume the statement is true for the algebra $\mathcal{Z}_{n-1,m-1}$, so $\mathcal{Z}_{n-1,m-1}=N_{n-1,a,s}$ with $a=[\frac{2(n-1)}{n-m+1}]$ and $s=\frac a 2 ((a-1)(n-1)-(a+1)(m-1)+a-1)+1=((a+1) (\frac {n-m+1} 2) -n)a+1$. Since $a+1 > \frac{2(n-1)}{n-m+1}$, we have $(a+1) (\frac {n-m+1} 2) -n >-1$, or equivalently $$(a+1) (\frac {n-m+1} 2) -n \geq 0.$$ Therefore $s \leq a$ if and only if and $s=1$ if and only if $(a+1) (\frac {n-m+1} 2) -n  =0$ if and only if $a+1 = \frac{2n}{n-m+1}$.

    Let $\mathcal{Z}_{n,m}=N_{n,a',s'}$. If $s=1$, then by the recursive formulas we have $$a'=a+1=\frac{2n}{n-m+1}=[\frac{2n}{n-m+1}]$$ and
    \begin{equation*}
    \begin{split}
    s'=s-a+n-1=-a+n &=(-1+(\frac {n-m+1} 2))(a+1)+1\\ &=((a+2) (\frac {n-m+1} 2)-1 -n)(a+1)+1\\ &=((a'+1) (\frac {n-m+1} 2)-1 -n)a'+1\\ &= \frac {a'} 2 ((a'-1)n-(a'+1)m+a'-1)+1.
    \end{split}
    \end{equation*}
    If $s>a$, then by the recursive formulas $$a'=a=[\frac{2(n-1)}{n-m+1}]=[\frac{2n}{n-m+1}],$$ since $2 \mid (n-m+1)$ and $\frac{2n}{n-m+1} \neq a+1$, and
    \begin{equation*}
    \begin{split}
    s'=s-a+n-1 &=\frac a 2 ((a-1)(n-1)-(a+1)(m-1)+a-1)+1-a+n-1\\ &=\frac {a'} 2 ((a'-1)n-(a'+1)m+a'-1)+1.
    \end{split}
    \end{equation*}
    This concludes the induction step and the statement is proven for $n-m$ odd.

\item Let $2 \leq m \leq n-2$ and assume $n-m$ is even. Assume the statement is true for the algebra $\mathcal{Z}_{n-1,m-1}$, so $\mathcal{Z}_{n-1,m-1}=N_{n-1,a,s}$ with $a=[\frac{2(n-2)}{n-m}]$ and $s=\frac a 2 ((a-1)(n-1)-(a+1)(m-1)+2)=((a+1) (\frac {n-m} 2) +2 -n) a$. Since $a+1 > \frac{2(n-2)}{n-m}$, we have $(a+1) (\frac {n-m} 2) +2 -n >0$, or equivalently $$(a+1) (\frac {n-m} 2) +2 -n \geq 1.$$ Therefore $s \geq a$, and $s=a$ if and only if $(a+1) (\frac {n-m} 2) +2 -n =1$ if and only if $a+1 = \frac{2(n-1)}{n-m}$.

    Let $\mathcal{Z}_{n,m}=N_{n,a',s'}$. If $s=a$, then by the recursive formulas we have $$a'=a+1=\frac{2(n-1)}{n-m}=[\frac{2(n-1)}{n-m}]$$ and
    \begin{equation*}
    \begin{split}
    s'=s-a+n-1=n-1 &=(\frac {n-m} 2) (a+1)\\ &=((a+2) (\frac {n-m} 2) +1 -n) (a+1)\\ &=((a'+1) (\frac {n-m} 2) +1 -n) a'\\ &=\frac {a'} 2 ((a'-1)n-(a'+1)m+2).
    \end{split}
    \end{equation*}
    If $s>a$, then by the recursive formulas $$a'=a=[\frac{2(n-2)}{n-m}]=[\frac{2(n-1)}{n-m}],$$ since $2 \mid (n-m)$ and $\frac{2(n-1)}{n-m} \neq a+1$, and
    \begin{equation*}
    \begin{split}
     s'=s-a+n-1 &=\frac a 2 ((a-1)(n-1)-(a+1)(m+1)-a+n-1\\ &=\frac {a'} 2 ((a'-1)n-(a'+1)m+2).
    \end{split}
    \end{equation*}
    This concludes the induction step and the statement is proven for $n-m$ even.
\end{enumerate}
\end{proof}

As a consequence we get the following corollary, where $\dim_K (A)$ denotes the vector space dimension of an algebra $A$.

\begin{corollary}
If $n-m$ is even, then $$\dim_K (\mathcal{Z}_{n,m})=\dim_K (\mathcal{Z}_{n-1,m}) +2.$$
\end{corollary}

\begin{proof}
Let $\mathcal{Z}_{n,m}=N_{n,a,s}$ and $\mathcal{Z}_{n-1,m}=N_{n-1,a',s'}$. We have $a'=[\frac{2(n-1)}{n-m}]=a$. Since $\dim_K (\mathcal{Z}_{n,m})=(a+1)n-s$ and $\dim_K (\mathcal{Z}_{n-1,m})=(a+1)(n-1)-s'$, we get
\begin{equation*}
\begin{split}
\dim_K (\mathcal{Z}_{n,m})- &\dim_K (\mathcal{Z}_{n-1,m}) =(a+1)n-s-((a+1)(n-1)-s')\\
&=a+1-s+s'\\
&=a+1-\frac a 2 ((a-1)n-(a+1)m+2)+\frac a 2 ((a-1)(n-1)-(a+1)m+a-1)+1\\
&=2.
\end{split}
\end{equation*}
\end{proof}

We discuss a possible alternative description of the algebras $\mathcal{Z}_{n,m}$ in the next section.

\section{D1-Nakayama algebras with finite Gorenstein dimension}

We gave a full classification of D1-Nakayama algebras with finite global dimension in the previous section. In this section we focus on the Gorenstein dimension of D1-Nakayama algebras.
We need two results from the literature that we recall first.

\begin{proposition}\label{domdimnak}
Let $r \geq 1$.
Let $B$ be a selfinjective Nakayama algebra with Loewy length $w$ and $n$ simple modules. Let
$M=\bigoplus\limits_{i=0}^{n-1}{e_i B} \oplus \bigoplus\limits_{i=1}^{r}{e_{x_i} B / e_{x_i} J^{w-1}}$ with the $x_i$ different for all $1 \leq i \leq r$ and $A=\End_B(M)$. Then every Morita algebra that is also a Nakayama algebra is isomorphic to some algebra of this form $A$, and for the dominant dimension one has
$$\domdim(A) = \inf \{ k \geq 1 \mid \exists x_i , x_j \colon x_j + w-1 \equiv_n x_i + [[\frac{k+1}{2}]] w-g_k  \}+1.$$
Here, we set $g_k=1$, if $k$ is even, and $g_k=0$, if $k$ is odd. Furthermore, we have that $r$ is the number of indecomposable projective non-injective $A$-modules.
\end{proposition}

\begin{proof}
See \cite[Proposition 3.2]{Mar}.
\end{proof}

The second result involves the $M$-resolution dimension of a module, where $M$ is a given generator-cogenerator of the module category, see \cite{CheKoe} for a definition. For our purposes the important property is that the $M$-resolution dimension of a module $X$, denoted $M \! \operatorname{-resdim}(X)$, is equal to zero if and only if $X$ is a direct sum of direct summands of $M$, and $M \! \operatorname{-resdim}(X)>0$ otherwise.
We use the common notation $\tau_{z+1}=\tau\Omega^{z}$, introduced by Iyama (see \cite{Iya}).

\begin{proposition}
\label{CheKoetheorem}
Let $B$ be a finite dimensional algebra and $M$ a generator-cogenerator of $\mod B$, and define $A \coloneqq \End_B(M)$.
Let $A$ have dominant dimension $z+2$, with $z \geq 0$.
Then, for the right injective dimension of $A$ the following holds.
$$\id (A_A)=z+2 + M \! \operatorname{-resdim}(\tau_{z+1}(M)\oplus D(B)).$$
\end{proposition}

\begin{proof}
See \cite[Proposition 3.11]{CheKoe}.
\end{proof}

The Gorenstein symmetry conjecture, which says that the right and left injective dimension of the regular module always coincide, holds true for all algebras of finite finitistic dimension by \cite[Proposition 6.10]{AusRei}.
Since we deal in this article only with representation-finite algebras, the Gorenstein symmetry conjecture is true for those algebras (since representation-finite algebras always have finite finitistic dimension), and thus it is enough to look at the right injective dimension in the following.
By Proposition \ref{domdimnak}, a D1-CNakayama algebra $A$ is a Morita algebra if and only if it is isomorphic to an algebra of the form $\End_B(B \oplus P/\soc(P))$, where $B$ is a selfinjective Nakayama algebra and $P$ an indecomposable projective $B$-module.

\begin{theorem}
Let $w >2$ and $n \geq 3$. Let $B$ be a selfinjective Nakayama algebra with Loewy length $w$ and $n$ simple modules. Let $A \coloneqq \End_B(B \oplus P/S)$ for some indecomposable projective module $P$ with $\soc(P)=S$.
The following are equivalent.
\begin{enumerate}
\item $A$ is Gorenstein.
\item The dominant dimension of $A$ is even.
\item $w$ is a unit in $\mathbb{Z}/n\mathbb{Z}$.
\end{enumerate}
If the equivalent statements are true, then $$\Gdim(A)= \domdim(A)=2 \cdot \inf \{ h \geq 0 \mid hw+1 \equiv_n 0 \}+2.$$
\end{theorem}

\begin{proof}
We use Propositions \ref{domdimnak} and \ref{CheKoetheorem} to prove this theorem.
By symmetry we may assume $S=e_0 B/e_0 J$.
In the case of just one indecomposable non-projective summand, the formula for the dominant dimension in Proposition \ref{domdimnak} simplifies to $$\domdim(A)=\inf \{ k \geq 1 \mid w-1 \equiv_n [[\frac{k+1}{2}]]w-g_k  \}+1.$$
Splitting into even and odd cases, this can be simplified further to give $\domdim(A)=\min \{ d_e , d_o \}$, where $$d_o=2 \cdot \inf \{ h \geq 1 \mid hw \equiv 0  \}+1$$ and $$d_e=2 \cdot \inf \{ h \geq 0 \mid hw + 1 \equiv 0 \}+2.$$

Now assume $\domdim(A)=d_e$ is even. Then one has $hw \equiv -1$ for some $h$, and therefore $w$ is a unit in $\mathbb{Z}/n\mathbb{Z}$. In case the $\domdim(A)=d_o$ is odd, one has $hw \equiv 0$ for some $h$, and so $w$ is a zero-divisor and can not be a unit. This shows the equivalence of (2) and (3).

Now we show that (2) implies (1). Assume that the dominant dimension of $A$ is even.
This means $\domdim(A)=d_e= 2 \cdot \inf \{ h \geq 0 \mid h \equiv -w^{-1} \}+2$.
Note that for a general indecomposable $B$-module $e_i B/e_i J^l$ with $1 \leq l \leq w-1$, we have $\Omega^2(e_i B /e_i J^l) =e_{i+w}B/e_{i+w}J^l$ and $\tau(e_i B /e_i J^l)=e_{i+1} B /e_{i+1} J^l$.
Then $$\tau_{d_e-1} (e_0 B/ e_0 J)=\tau(\Omega^{d_e-2}(e_0 B/ e_0 J))=\tau(e_{w\frac{d_e-2}{2}}B/e_{w\frac{d_e-2}{2}} J)=e_{w\frac{d_e}{2}}B/e_{w\frac{d_e}{2}} J=e_{0}B/e_{0} J,$$ since by the definition of $d_e$ we have $w\frac{d_e}{2} =0$ (modulo $n$). Thus for $M=B \oplus P/S$ we have $$M \! \operatorname{-resdim}(\tau_{d_e-1}(M)\oplus D(B))=0.$$ By Proposition \ref{CheKoetheorem} we have $\Gdim(A)=\domdim(A)$ and $A$ is Gorenstein.

What is left to show is (1) implies (2), that $A$ is never Gorenstein in case the dominant dimension of $A$ is odd. So assume now that $\domdim(A)=d_o=2 \cdot \inf \{ h \geq 1 \mid hw \equiv 0  \}+1$.
Note that for an indecomposable $B$-module of the form $N=e_i B/ e_i J^l$ with $1 \leq l \leq w-1$,  we have that $N$ and $\Omega^r(N)$ have vector space dimension $l$ for all even $r$ (by the above calculation of $\Omega^2$), while $\Omega^1(N)$ and thus also $\Omega^q(N)$ have vector space dimension $w-l$ for all odd $q$.
Also $\tau(N)$ always has the same vector space dimension as $N$.
So $\tau(\Omega^{d_o-2}(e_0 B/e_0 J))$ has vector space dimension $w-1 \neq 1$, and thus $\tau(\Omega^{d_o-2}(e_0 B/e_0 J))$ can not be isomorphic to $e_0 B/ e_0 J$. So for $M=B \oplus P/S$ we have $$M \! \operatorname{-resdim}(\tau_{d_o-1}(M)\oplus D(B))>0.$$
This means that by Proposition \ref{CheKoetheorem}, the algebra $A$ has Gorenstein dimension strictly larger than the dominant dimension. But since $A$ is a representation-finite algebra of defect $1$, by Theorem \ref{findimtheorem} the dominant dimension equals the finitistic dimension, and thus $A$ would have Gorenstein dimension equal to the dominant dimension in case the Gorenstein dimension is finite. Hence the Gorenstein dimension of $A$ is infinite.
\end{proof}

\begin{remark}
We excluded the case $w=2$ in the previous theorem. In the case $w=2$ when $B$ is a selfinjective Nakayama algebra of Loewy length 2, it is easy to show that the algebra $\End_B(B \oplus P/S)$ has Kupisch series $[2,2,2,\dots,2,3]$ and global dimension $n+1$ and thus is always Gorenstein.
\end{remark}

\begin{remark} \label{remarkcluster}
One can show that the algebra $A=\End_B(B \oplus P/S)$ as in the previous theorem has finite global dimension (which is equivalent to $B \oplus P/S$ being a cluster-tilting module) if and only if $w=2$ or $w$ divides $n+1$. A proof will be given in forthcoming work \cite{MSTZ} using determinants.
\end{remark}

We saw in the previous theorem that for D1-Nakayama algebras that are Morita algebras the Gorenstein dimension is finite if and only if the dominant dimension is even. This is in fact true for all D1-Nakayama algebras with infinite global dimension as we show in the following.
We collect several results from the literature that we need in the proof.

\begin{lemma} \label{sinklemma}
Let $A$ be a non-selfinjective CNakayama algebra.
Then the number of sources in the resolution quiver and the number of sources in the injective resolution quiver of $A$ are equal to the defect of $A$.
\end{lemma}

\begin{proof}
This is explained in \cite[Remark 1]{Rin}.
\end{proof}

Recall that a module $M$ is called \emph{periodic} in case $\Omega^l(M) \cong M$ for some $l >0$.
The \emph{$\phi$-dimension} of a Nakayama algebra $A$ is defined as the smallest integer $k$ such that the following holds for any indecomposable $A$-module $M$:
$\Omega^k(M)=0$ or $\Omega^k(M)$ is periodic.
We note that for general Artin algebras the definition of the $\phi$-dimension is more complicated and we used here an equivalent characterisation for Nakayama algebras that is a consequence of the results in sections 3.2 and 3.3 of \cite{Sen2}.
We remark that the methods in \cite{Sen} and \cite{Sen2} have also been used already in \cite{CY} for Nakayama algebras and were recently generalised and used in \cite{M} for higher Nakayama algebras.
\begin{theorem} \label{phidim}
Let $A$ be a Nakayama algebra with infinite global dimension.
\begin{enumerate}
\item $\phi(A)$ is an even integer.
\item $1 \geq \phi(A) - \findim(A) \geq 0$.
\end{enumerate}
\end{theorem}

\begin{proof}
\begin{enumerate}
\item See \cite[Theorem 1]{Sen}.
\item See \cite[Theorem (A)]{Sen2}.
\end{enumerate}
\end{proof}

\begin{theorem} \label{D1Gor}
Let $A$ be a D1-Nakayama algebra of infinite global dimension.
Then the following are equivalent.
\begin{enumerate}
\item $A$ is a Gorenstein algebra.
\item $A$ is a minimal Auslander-Gorenstein algebra.
\item $A$ has even dominant dimension.
\item $A$ has even finitistic dimension.
\end{enumerate}
\end{theorem}

\begin{proof}
The equivalence of (1) and (2) and the equivalence of (3) and (4) are a consequence of Theorem \ref{findimtheorem}. We show that (1) is equivalent to (3) to finish the proof.

In case $A$ is Gorenstein with infinite global dimension, it has even Gorenstein dimension by \cite[Corollary 4.12]{Sen}, so (1) implies (3).

To prove (3) implies (1), now assume $A=N_{n,a,s}$ has even dominant dimension. We show that $A$ has finite Gorenstein dimension by using the resolution quiver of $A$. For each vertex $i$, there is an arrow $i \to f(i)$ in the resolution quiver, where $f(i)=i+c_i$ (modulo $n$). By Lemma \ref{Shentheorem}(2), the algebra $A$ is Gorenstein if and only if every cycle in the resolution quiver of $A$ is black. Since the unique red vertex of $A$ is $n-1$, we just have to show that $n-1$ does not lie on a cycle.
Note that all cycles in the resolution quiver have the same length and there are always arrows $0 \to 0+c_0=a$ and $n-1 \to n-1 +c_{n-1}=n-1+a+1=a$.
We assume to the contrary that $A$ is not Gorenstein and thus there is a cycle containing the arrow $n-1 \to a$. This cycle can not contain the vertex $0$ since the resolution quiver contains at most one cycle in a connected component.
We show that this leads to a contradiction to our assumption that the dominant dimension is even.
Let $2u$ be the dominant dimension of $A$.
The dominant dimension coincides with the finitistic dimension of $A$ since $A$ is a D1-Nakayama algebra. Because of the inequality $1 \geq \phi(A) - \findim(A) \geq 0$ and since $\phi(A)$ is an even integer (here we use our assumption that $A$ has infinite global dimension) by Theorem \ref{phidim}, we conclude that $\phi(A)=\findim(A)$.
Since $A$ has infinite Gorenstein dimension, the projective dimension of the unique indecomposable injective non-projective module $I \coloneqq e_sA/e_sJ^a$ is infinite, and thus $\Omega^{2u}(I)$ is a periodic module by the definition of the $\phi$-dimension.
The codominant dimension of $I$ is equal to $2u$, and thus the codominant dimension of $\Omega^{2u}(I)$ is equal to zero. Then there is a projective cover $e_0A \to \Omega^{2u}(I) \to 0$ since $e_0 A$ is the unique indecomposable projective non-injective $A$-module.
Since $\Omega^{2u}(I)$ is indecomposable, its minimal projective resolution is of the form $$\dots \to e_{f^2(j)}A \to e_{f^2(0)}A \to e_{f(j)}A \to e_{f(0)}A  \to e_j A \to e_0 A \to \Omega^{2u}(I) \to 0$$ for some $j$ according to \cite{Gus}. Since $\Omega^{2u}(I)$ is periodic, the vertex $0$ must lie on a cycle in the resolution quiver.
Thus we have arrived at a contradiction and $A$ must be Gorenstein.
\end{proof}

The next example shows that it is not in general true that a Nakayama algebra of infinite global dimension and even dominant dimension has finite Gorenstein dimension.

\begin{example}
The Nakayama algebra $A$ with Kupisch series $[ 4, 4, 5, 6, 5]$ has infinite Gorenstein dimension but the dominant dimension of $A$ is equal to $2$.
\end{example}

In the previous section we gave an explicit description of the algebras $\mathcal{Z}_{n,m}$ in terms of their Kupisch series. Since those algebras are higher Auslander algebras, one can also write them as the endomorphism ring $\End_A(M)$ of another Nakayama algebra $A$ with a cluster-tilting module $M$.
In case $A$ is selfinjective, a very nice such representation is possible as remarked in Remark \ref{remarkcluster}. In case $A$ is not selfinjective, one necessarily has that $A$ is again a D1-CNakayama algebra and $M$ has to be the basic version of the module $A \oplus D(A)$. This leads directly to the following open question, whose solution would give new representations of the algebras $\mathcal{Z}_{n,m}$ for $m \geq 1$.

\begin{question}
Is there a nice description of the D1-CNakayama algebras $A$ such that the module $A \oplus D(A)$ is cluster-tilting?
\end{question}

\section{Improved inequality for the global dimension of Nakayama algebras}

We only deal with CNakayama algebras in this section, since the results are trivial for LNakayama algebras because they have global dimension at most $n-1$ when they have $n$ simple modules.
Let $A$ be a CNakayama algebra, and let $\mathcal S$ denote a complete set of representatives of the isomorphism classes of simple $A$-modules. Following \cite{Mad}, we define a function $\psi \colon \mathcal S \to \mathcal S$ by
$\psi (S) \cong (\tau^{-1})^{w(S)} S$, where $w(S)$ is the length of the injective envelope of $S$, for each $S \in \mathcal S$. In other words, for $S, S' \in \mathcal S$, we have $w(S)=S'$ if and only if in the injective resolution quiver of $A$ there is an arrow from the vertex corresponding to $S$ to the vertex corresponding to $S'$. We say that $S \in \mathcal S$ is \emph{$\psi$-regular} if $\psi^r(S)=S$ for some $r \geq 1$, or equivalently the vertex corresponding to $S$ lies on a cycle in the injective resolution quiver of $A$.

By \cite{Mad}, a CNakayama algebra has finite global dimension if and only if there is a simple module of even projective dimension. In \cite{MM}, we proved the following.

\begin{theorem} \label{MMtheorem}
Let $A$ be a CNakayama algebra with $n$ simple modules and a simple module $S$ of even projective dimension.
Choose $m$ minimal such that there is a simple module with projective dimension $2m$. Then the global dimension of $A$ is bounded above by $n+m-1$.
\end{theorem}

In this section we prove that the bounds in Theorem \ref{MMtheorem} are uniquely attained. Since we are dealing with CNakayama algebras, the possible values for $m$ are $1 \leq m \leq n-1$.

\begin{proposition} \label{gldimD1propo}
Let $A=N_{n,a,s}$ be a D1-CNakayama algebra with finite global dimension. Then $\gldim(A)=n+m-1$, where $m$ is the smallest number such that there is a simple module with even projective dimension $2m$.
Furthermore, in this case $\pd(S_{s-1})=2m$.
\end{proposition}

\begin{proof}
Since $A$ has finite global dimension, by \cite[Theorem 3.3]{Mad} the injective resolution quiver of $A$ has a unique component. Since $A$ is a D1-CNakayama algebra, the unique component has a unique source by Lemma \ref{sinklemma}.

By \cite[Lemma 3.1(b)]{Mad}, for all $S, S' \in \mathcal S$ we have $\psi(S')=S$ if and only if $S'$ is a composition factor of $\Omega^2 (S)$.
Let $U=S_{n-1}$, which by Lemma \ref{OmegaTwo} is the unique simple $A$-module such that $\Omega^2 (U)=0$, in other words $\pd (U) =1$. Then in the injective resolution quiver the vertex corresponding to $U$ must be the unique source. Let $d$ be the number of simple $A$-modules that are not $\psi$-regular. Then $\{U, \psi(U), \dots, \psi^{d-1}(U) \}$ is a complete list of simple modules with odd projective dimension by \cite[Theorem 3.3]{Mad}.

Let $V=S_{s-1}$, which by Lemma \ref{OmegaTwo} is the unique simple $A$-module $V$ such that $\Omega^2 (V)$ has length two. Then in the injective resolution quiver the vertex corresponding to $V$ must be the unique vertex with indegree two, and $\{V, \psi(V), \dots, \psi^{n-d-1}(V) \}$ is a complete list of simple modules with even projective dimension.

By Lemma \ref{OmegaTwo}, the modules $\Omega^2 (\psi^i(U))$ are simple for $0< i \leq d-1$, and then by repeated use of \cite[Lemma 3.1(b)]{Mad} we get $U \cong \Omega^{2i} (\psi^i(U))$ for $0< i \leq d-1$. So $$\pd (\psi^i(U))=2i +\pd(U)=2i+1$$ for $0 \leq i \leq d-1$. Therefore the maximal projective dimension among the simple modules with odd projective dimension is $\pd (\psi^{d-1}(U))=2d-1$. Using the fact that if the global dimension of a finite dimensional algebra is $g$, then there exist a simple module of projective dimension $g$ and a simple module of projective dimension $g-1$, we can conclude that $\gldim(A)=2d-1$ or $\gldim(A)=2d$.

Similarly by Lemma \ref{OmegaTwo}, the modules $\Omega^2 (\psi^i(V))$ are simple for $0< i \leq n-d-1$, and again by repeated use of \cite[Lemma 3.1(b)]{Mad} we get $V \cong \Omega^{2i} (\psi^i(V))$ for $0< i \leq d-1$. So $$\pd (\psi^i(V))=2i +\pd(V)$$ for $0 \leq i \leq n-d-1$. Therefore the minimal projective dimension among the simple modules with even projective dimension is $\pd (V)$, so $\pd (V)=2m$, and the maximal projective dimension among the simple modules with even projective dimension is $\pd (\psi^{n-d-1}(V))=2n-2d-2+2m$.

Assume $\gldim(A)=2d-1$, an odd number. Then the maximal projective dimension among the simple modules with even projective dimension is $$\gldim(A) -1= 2n-2d-2+2m= 2n +2m -\gldim(A)-3.$$ Solving for $\gldim(A)$ we get $\gldim(A)=n+m-1$.

Assume $\gldim(A)=2d$, an even number. Then the maximal projective dimension among the simple modules with even projective dimension is $$\gldim(A)= 2n-2d-2+2m= 2n +2m -\gldim (A) -2.$$ This also implies $\gldim(A)=n+m-1$.
\end{proof}

In section \ref{sectd1} we used the notation $\mathcal{Z}_{n,m}$ to denote the unique D1-CNakayama algebra with $n$ simple modules and global dimension $n+m-1$ for $1 \leq m \leq n-1$. As a consequence of Proposition \ref{gldimD1propo}, the minimal projective dimension among the simple $\mathcal{Z}_{n,m}$-modules with even projective dimension is $2m$. So for $1 \leq m \leq n-1$, the bound in Theorem \ref{MMtheorem} is attained by the algebras $\mathcal{Z}_{n,m}$. We next show this is unique among CNakayama algebras.

\begin{theorem} \label{gldimboundunique}
Let $A$ be a CNakayama algebra with $n$ simple modules and finite global dimension.
Choose $m$ minimal such that there is a simple module with projective dimension $2m$ and assume $A$ has global dimension $n+m-1$. Then $A$ is isomorphic to $\mathcal{Z}_{n,m}$, and hence $A$ is a higher Auslander algebra.
\end{theorem}

\begin{proof}
We show that in case $A$ is an algebra as in the theorem with global dimension $n+m-1$, it has to be a D1-Nakayama algebra.
As a D1-Nakayama algebra with finite global dimension, it has to be isomorphic to $\mathcal{Z}_{n,m}$ as we saw in Theorem \ref{uniqueness proof}.
Assume first that $A$ has global dimension $n+m-1$ and a simple module with projective dimension $2m$ and $m$ minimal such that such a simple module exists.
We show that the injective resolution quiver of $A$ has exactly one source and thus $A$ is a D1-Nakayama algebra by Lemma \ref{sinklemma}.
By assumption we have $$2 \cdot \gldim(A)=2n+2m-2.$$ Let $d$ be the number of simple $A$-modules that are not $\psi$-regular. According to the proof of Theorem 2.2 in \cite{MM}, we have $$\gldim(A) \leq 2n+2m-2d-1.$$ In conclusion, $$\gldim(A) =2n+2m-2-\gldim(A) \geq 2n+2m-2-(2n+2m-2d-1)= 2d-1.$$

Let $d'$ be the smallest integer such that $\psi^{d'}(S)$ is $\psi$-regular for all $S \in \mathcal S$. Then $d' \leq d$ and $\gldim A \leq 2d'$ by the main argument of \cite{Gus}. The inequalities $$2d-1 \leq \gldim A \leq 2d' \leq 2d$$ imply that $d'=d$. So there exists $S' \in \mathcal S$ such that $\{S',\psi(S'),\dots,\psi^{d-1}(S')\}$ is a complete list of simple $A$-modules that are not $\psi$-regular. It follows that $S'$ is the unique source in the injective resolution quiver of $A$.
\end{proof}

\begin{corollary}
Let $A$ be a quasi-hereditary CNakayama algebra with $n$ simple modules.
Then $A$ has global dimension equal to $n$ if and only if it has Kupisch series [2,2,\dots,2,3]. Else we have $\gldim(A) \leq n-1$.
\end{corollary}

\begin{proof}
This is the special case of Theorem \ref{gldimboundunique} in the case $m=1$, since a CNakayama algebra is quasi-hereditary if and only if it has a simple modules of projective dimension two by \cite[Proposition 3.1]{UY}.
\end{proof}

\section{A conjecture for the existence of higher Auslander algebras among Nakayama algebras}

We end this article with a short section where we discuss the existence of higher Auslander algebras among Nakayama algebras.
Given a fixed finite quiver $Q$ and a field $K$ we define the \emph{spectrum} $\zeta_{(Q,K)}$ of the tuple $(Q,K)$ as $$\zeta_{(Q,K)} \coloneqq \{ \gldim(KQ/I) \mid KQ/I \text{ is a higher Auslander algebra for some admissible ideal } I \}.$$
In general it seems very complicated to calculate $\zeta_{(Q,K)}$ for non-trivial choices of $Q$. Since Nakayama algebras are one of the most basic algebras in representation theory, the following questions are natural.

\begin{question}
Let $Q_{L_n}$ be the quiver of a linear oriented line with $n$ points and $Q_{C_n}$ be the quiver of a linear oriented cycle with $n$ points as in the preliminaries.
\begin{enumerate}
\item What is $\zeta_{(Q_{L_n},K)}$ for $n \geq 2$?
\item What is $\zeta_{(Q_{C_n},K)}$ for $n \geq 2$?
\end{enumerate}
\end{question}

We give one example that has been obtained with the help of a computer:

\begin{example}
For $n=9$, we have $$\zeta_{(Q_{L_9},K)}= \{ 2, 3, 4, 5, 8 \}$$ and $$\zeta_{(Q_{C_9},K)}= \{3, 4, 5, 6, 7, 9, 10, 11, 12, 13, 14, 15, 16\}.$$
\end{example}

In general it seems very complicated to answer the previous question for a given $n$. However, we noted with the help of the computer that $$\zeta_{(Q_{L_n},K)} \cup \zeta_{(Q_{C_n},K)} = \{2,3, \dots ,2n-2 \}$$ for $2 \leq n \leq 14$.

Motivated by this, we end this article with the following conjecture on the existence of higher Auslander algebras for Nakayama algebras.

\begin{conjecture}
For every $n \geq 2$ and $k$ with $2 \leq k \leq 2n-2$ there exists a connected Nakayama algebra with $n$ simple modules that is a higher Auslander algebra with global dimension $k$.
\end{conjecture}

Note that $2n-2$ is also the bound on the global dimension of a Nakayama algebra with $n$ simple modules and thus the conjecture says that essentially for all possible values of the global dimension ($\geq 2$) of a Nakayama algebra there exists a higher Auslander algebra with this value for the global dimension.
We proved the conjecture in this article for $k$ with $n-1 \leq k \leq 2n-2$ and we note that the conjecture is easy to verify for $k=2$ and $k=3$. We verified the conjecture with a computer for $n \leq 14$.

\end{document}